\documentclass[pdflatex,sn-mathphys-num]{sn-jnl}% Math and Physical Sciences Numbered Reference Style
\usepackage{graphicx}%
\usepackage{multirow}%
\usepackage{amsmath,amssymb,amsfonts}%
\usepackage{amsthm}%
\usepackage{mathrsfs}%
\usepackage[title]{appendix}%
\usepackage{xcolor}%
\usepackage{textcomp}%
\usepackage{manyfoot}%
\usepackage{booktabs}%
\usepackage{algorithm}%
\usepackage{algorithmicx}%
\usepackage{algpseudocode}%
\usepackage{listings}%
\numberwithin{equation}{section}

\usepackage{tikz}

\theoremstyle{plain}
\newtheorem{theorem}{Theorem}[section]
\newtheorem{lemma}[theorem]{Lemma}
\newtheorem{corollary}[theorem]{Corollary}
\newtheorem{proposition}[theorem]{Proposition}
\newtheorem{conjecture}[theorem]{Conjecture}

\newtheorem{algorithm }[theorem]{Algorithm}

\theoremstyle{definition}
\newtheorem{definition}[theorem]{Definition}

\newtheorem{problem}[theorem]{Problem}

\newtheorem{remark}[theorem]{Remark}

\newtheorem{case[theorem]}{Case}

\begin{document}

\title[Mapping properties of the $\mathcal{S}$-operator]{Mapping properties of the $\mathcal{S}$-operator}


\author[1]{\fnm{Hunseok} \sur{Kang}}\email{hunseok.kang@aum.edu.kw}
\equalcont{These authors contributed equally to this work.}

\author*[2]{\fnm{Doowon} \sur{Koh}}\email{koh131@chungbuk.ac.kr}
\equalcont{These authors contributed equally to this work.}

\author[2]{\fnm{Changhun} \sur{Yang}}\email{chyang@chungbuk.ac.kr}
\equalcont{These authors contributed equally to this work.}

\affil[1]{\orgdiv{College of Engineering and Technology}, \orgname{American University of the Middle East}, \orgaddress{\city{Egaila}, \postcode{54200}, \country{Kuwait}}}

\affil[2]{\orgdiv{Department of Mathematics}, \orgname{Chungbuk National University}, \orgaddress{\city{Cheongju}, \postcode{28644}, \country{Korea}}}


\abstract{In this paper, we study the $\ell^p\to \ell^r$ estimates for the $\mathcal{S}$-operator arising in restriction problems for spheres over finite fields. We establish a necessary and sufficient condition for the boundedness of the $\mathcal{S}$-operator. Furthermore, we investigate this problem under certain restrictions on test functions. In particular, we address the sharp results  when test functions are restricted to radial functions.}

\keywords{The $\mathcal{S}$-operator, Finite field, Boundedness}


\pacs[MSC Classification]{42B05, 43A32, 43A15}

\maketitle

\section{Introduction}%\label{sec1}

\subsection{Historical background of finite field restriction theory}

The restriction conjecture in harmonic analysis, originally formulated by Elias Stein \cite{St78} in the 1970s, has been one of the most influential problems in modern mathematical analysis. In the classical Euclidean setting, the restriction problem concerns the boundedness of operators that restrict the Fourier transform to curved submanifolds such as spheres, paraboloids, and cones. The fundamental question asks for the optimal range of exponents $(p,r)$ such that the restriction operator $R_M: L^p(\mathbb{R}^d) \to L^r(M, d\sigma)$ is bounded, where $M$ is a smooth manifold and $d\sigma$ is the induced measure. Since its close connection to other fundamental problems in harmonic analysis, including the Kakeya problem and the Bochner-Riesz problem, was established, this problem has attracted significant research interest and is currently under active investigation. Nevertheless, it remains open except for the two-dimensional case, which was solved by Fefferman \cite{Fe70}. For recent progress on the restriction conjecture in Euclidean space, we refer the reader to the papers \cite{To75, Bo91, Wo01, Ta03, BD15, Gu16, GIOW20, WW24} and the references therein.

The work of Mockenhaupt and Tao \cite{MT04} in 2004 revolutionized this field by introducing finite field analogues of classical restriction problems. Let $\mathbb F_q^n , n\ge 2,$ be an $n$-dimensional vector space over a finite field $\mathbb F_q$ with $q$ elements. For each finite field $\mathbb F_q,$ we let $V_q$ be a fixed variety in $\mathbb F_q^n.$ For $1\le p, r\le \infty,$ we define $R_{V_q}(p\to r)$ as the smallest number such that the restriction estimate
$$ \|\widehat{f}\|_{L^r(V_q)} \le R_{V_q} (p\to r) \|f\|_{\ell^p(\mathbb F_q^n)}$$
holds for all functions $f: \mathbb F_q^n \to \mathbb C,$ where $\widehat{f}$ denotes the Fourier transform of  $f$ defined in Section \ref{secPre}. Here, and throughout the paper, $L^r(V_q)$ denotes the space of functions on $V_q$ with respect to the normalized counting measure, while $\ell^p(\mathbb{F}_q^n)$ denotes the space of functions on $\mathbb{F}_q^n$ with respect to the counting measure.

The restriction problem for the variety  $V_q$ asks for all exponents $1\le p, r\le \infty$ such that  $R_{V_q}(p\to r) \le C$ for some constant $C>0$ independent of $q.$ In other words,  we want to find all pairs $(p, r)$ with $1\le p, r \le \infty$ such that $R_{V_q}(p\to r)\lesssim 1.$ Here and throughout the paper, the notation $A \lesssim B$ means that there exists an absolute constant $C > 0$ such that $A \leq C B$. More precisely, we write $A \lesssim B$ if there exists a constant $C > 0$, independent of the relevant parameter $q$, such that the inequality $A \leq CB$ holds. The notation $A \gtrsim B$ is similarly defined to mean $B \lesssim A$, and $A \sim B$  means that both $A \lesssim B$ and $B \lesssim A$ hold simultaneously.

Mockenhaupt and Tao \cite{MT04} addressed reasonably good results on low-dimensional cones and paraboloids. We recall that the paraboloid $P$ and the cone $C$ in $\mathbb F_q^d$ are defined by
$$ P:=\left\{(x_1, \ldots, x_d) \in \mathbb F_q^d: x_d= \sum_{k=1}^{d-1} x_k^2\right\}$$
and $$ C:=\left\{(x_1, \ldots, x_d) \in \mathbb F_q^d: x_{d-1} x_d = \sum_{k=1}^{d-2} x_k^2\right\}.$$

In particular, they resolved the restriction conjecture for the paraboloid in two dimensions and the cone in three dimensions by showing that  $R_{V_q}(2\to 4) \lesssim 1$ holds.  In higher dimensions, their results on these two surfaces have been steadily improved with much interest. For example, results on paraboloids in three dimensions have been consistently improved by Lewko \cite{LL10, Le13, Le14, Le20, Le24} and Rudnev and Shkredov \cite{RS18}. On the other hand,  Iosevich, Lewko, and the second listed author have improved restriction estimates for high dimensional paraboloids (see, for instance, \cite{IK09, Ko20, IKL20}).  For cones, new research results have been addressed by  Lee, Pham and the second listed author \cite{KLP22}. In particular, they established the restriction conjecture for cones in $4$-dimensional space when $-1$ is not a square.

\subsection{Development of spherical restriction theory over finite fields}
Following the foundational work of Mockenhaupt and Tao \cite{MT04}, the study of restriction phenomena for specific algebraic varieties over finite fields has flourished. Iosevich and Koh  made significant early contributions, beginning with their 2008 work \cite{IK08} on restriction theorems for nondegenerate quadratic surfaces. Their systematic approach established sharp results for two-dimensional cases and achieved the Tomas-Stein exponent in higher dimensions, utilizing sophisticated estimates for the Kloosterman sums. The case of spheres has proven particularly rich and challenging. We recall that the sphere $S_j^{d-1}$ with a radius $j\in \mathbb F_q$ in $\mathbb F_q^d$ is defined by
$$S_j^{d-1}:=\left\{ (x_1, \ldots, x_d)\in \mathbb F_q^d: \sum_{k=1}^d x_k^2 =j\right\}.$$

Unlike the relatively well-understood paraboloid case, spherical restriction problems over finite fields involve more complex character sum estimates and exhibit different behavior depending on the radius and dimension. The work of Iosevich and Koh \cite{IK10} on spherical restriction theorems laid crucial groundwork, and a clearer statement of the restriction conjecture for spheres and further new results have recently been addressed by Koh, Pham, and Vinh \cite{KPV}. However, many fundamental questions still remain open, and the known results are much fewer compared to paraboloids and cones, making it considered a very difficult problem. 

\subsection{Application of finite spherical restriction problems to the Erd\H{o}s-Falconer distance problem}
While restriction problems on spheres have been more challenging and have attracted less attention than those on paraboloids or cones, they are proving to be far more valuable in terms of practical applicability. Although they can be applied to various problems, the most remarkable aspect is that restriction estimates for spheres can yield results for the Erd\H{o}s-Falconer distance problem introduced by Iosevich and Rudnev \cite{IR07}.  Let $E\subset \mathbb{F}_q^d.$ We recall that the Erd\H{o}s-Falconer distance problem, one of the most interesting and challenging problems in finite field discrete geometry, seeks to determine the smallest exponent $\alpha$ such that if $|E|\ge C q^\alpha$ for a sufficiently large constant $C$ independent of $q$, then $|\Delta(E)| \gtrsim q$,  where the distance set $\Delta(E)$  is defined by
$$ \Delta(E)=\left\{\|x-y\|:=\sum_{i=1}^d (x_i-y_i)^2: x, y\in E\right\}.$$

Iosevich and Rudnev \cite{IR07} showed that if $|E|>2q^{(d+1)/2},$ then $|\Delta(E)|=q.$ The sharpness of the exponent $(d+1)/2$ was established in \cite{HIKR10} for odd dimensions $d\ge 3$. In other words, the Erd\H{o}s-Falconer distance problem has been completely resolved for all odd dimensions $d \geq 3$. However, for even dimensions $d\ge 2$, it is conjectured that the exponent $(d+1)/2$ can be improved to $d/2$ (see  \cite{IR07, KKR25}).

\begin{conjecture}
Suppose that $d \geq 2$ is an even integer and $E \subset \mathbb{F}_q^d$. If $|E| \geq Cq^{d/2}$ for a sufficiently large constant $C$ independent of $q$, then $|\Delta(E)| \gtrsim q$.
\end{conjecture}
This conjecture remains open for all even dimensions.  For $d=2$ over general finite fields, it was established in \cite{CEHIK10, HLR, BHIPJR} that the $(d+1)/2$ (= $3/2$) exponent of Iosevich and Rudnev can be improved to $4/3$. Notably, when $\mathbb{F}_q$ is a prime field,  Murphy,  Petridis, Pham, Rudnev, and  Stevens \cite{MPPRS} achieved the same $5/4$ result that the authors \cite{GIOW20} had obtained for the Euclidean Falconer distance problem in two dimensions. However, no improvements over the $(d+1)/2$ exponent of Iosevich and Rudnev have been obtained for any higher even dimension $d \geq 4.$  To improve this result, we can exploit the connection between $L^2$ restriction estimates and the Erd\H{o}s-Falconer distance problem. For instance, Lemma 4.1 in \cite{CKP19} implies the following fact. 
\begin{proposition} [Lemma 4.1, \cite{CKP19}] \label{REFormula}
Suppose that there exists a constant $C$ independent of $q$ such that for all $j\in \mathbb F_q^*,$ it satisfies that $R_{S_j^{d-1}} \left(p\to 2\right)\leq C$ for some $1\le p\le \infty.$  Then for any $E\subset \mathbb F_q^d, d\ge2,$  we have
$$ |\Delta(E)| \gtrsim  \min\left\{q, ~~ \frac{|E|^{3-2/p}}{q^{d-1}}\right\}.$$ 
In particular, if $|E|\ge q^{\frac{dp}{3p-2}},$ then  $|\Delta(E)|\gtrsim q.$
\end{proposition}
By combining this proposition with the sharp restriction estimate \(R_{S_j^1}\left(\frac{4}{3} \to 2\right) \lesssim 1\), the authors in  \cite{CEHIK10} obtained the exponent \(4/3\), which improves upon the bound \((d+1)/2\) when \(d=2\).  In general odd dimensions $d \geq 3$ with $j \neq 0$, it was shown in \cite{IK08} that $p = \frac{2d+2}{d+3}$ is the sharp exponent for the bound $R_{S_j^{d-1}}(p \to 2) \lesssim 1$. This can be combined with Proposition \ref{REFormula} to recover the sharp exponent $(d+1)/2$ for the Erd\H{o}s--Falconer distance problem in odd dimensions $d \geq 3$. It is conjectured that in even dimensions $d \ge 4$, the value $p = \frac{2d+2}{d+3}$ from the sharp $L^2$ restriction theorem in odd dimensions can be improved to a larger exponent, but such an improvement remains open.  

Notice that  to break down the currently known $(d+1)/2$ Erd\H{o}s-Falconer distance result for even high dimensions $d \geq 4$, it would suffice to show from Proposition~\ref{REFormula} that $R_{S_j^{d-1}}\left(\frac{2d+2}{d+3} + \varepsilon \to 2\right) \lesssim 1$ for all non-zero $j$ and for some $\varepsilon > 0$. In other words, the following formula holds.
\begin{corollary}\label{epsilonIm} Let $d\ge 4$ be an even integer.  Suppose that the restriction estimate 
\begin{equation} \label{HopeR} R_{S_j^{d-1}}\left(\frac{2d+2}{d+3} + \varepsilon \to 2\right) \lesssim 1\end{equation}
holds for all non-zero $j$ and some $\varepsilon > 0$.  Then for any set $E\subset \mathbb F_q^d$ with $|E|\ge C q^{\frac{d+1}{2} - \frac{\varepsilon(d+3)^2}{8d + 6\varepsilon d + 18\varepsilon} },$  we have 
$$|\Delta(E)|\gtrsim q.$$
\end{corollary}
Proving the existence of $\varepsilon > 0$ satisfying the spherical restriction estimate \eqref{HopeR} is a very difficult problem, and currently no satisfactory method has been presented for this problem, whereas such existence is well known for the paraboloid in even dimensions $d\ge 4$ (for example, see \cite{MT04, IKL20, KPV}). The main reason for this is that the Fourier transform on a sphere is related to the Kloosterman sum, which is very difficult to handle, whereas the paraboloid is associated with the Gauss sum, which has an explicit closed form.

\subsection{Main problem and its motivation}
To provide some evidence that the finite field restriction conjecture for spheres is true, Kang and Koh \cite{KK14} considered the problem of $\ell^p \to L^r$ estimates for restriction operators on restricted test functions. In particular, in \cite{KK26} they showed that the conjectured $\ell^p\to L^2$ estimate holds when the test functions are homogeneous functions of degree zero. To prove this, they introduced  the $\mathcal{S}$-operator, defined below, and applied the boundedness of this operator along with restriction estimates for $j$-homogeneous varieties $H_j^d$ in $\mathbb F_q^{d+1},$ where we recall that for each non-zero $j\in \mathbb F_q^*,$
\[H_j^{d}=\{(x, x_{d+1})\in \mathbb F_q^d\times \mathbb F_q: \|x\|=jx_{d+1}^2\}.\]

\begin{definition} [The $\mathcal{S}$-operator]
Given a function $g: \mathbb F_q^d \to C,$  the $\mathcal{S}$-operator is defined by the relation
$$ \mathcal{S}g(m, m_{d+1}):= \frac{1}{q} \sum_{t\in \mathbb F_q^*} \chi(t m_{d+1}) g(tm),$$
where $m\in \mathbb F_q^d, m_{d+1}\in \mathbb F_q$  and $\chi$ denotes a nontrivial additive character of $\mathbb F_q.$
\end{definition} 

Beyond its applications to restriction theory (discussed below), the $\mathcal{S}$-operator is mathematically interesting in its own. Unlike standard operators in harmonic analysis, $\mathcal{S}$ maps functions on $\mathbb{F}_q^d$ to functions on $\mathbb{F}_q^{d+1}$.This dimension-changing operator exhibits sharp boundedness transitions requiring new analytical techniques, with dramatically different behavior for general versus radial functions. 

In this paper,  we study the $\ell^p\to \ell^s$ estimate for the $\mathcal{S}$-operator.
\begin{problem} [The $\mathcal{S}$-operator problem] For $1\le p, s\le \infty,$ let $\mathcal{S}(p\to s)$ denote the smallest number such that the estimate
$$ \|\mathcal{S}g\|_{\ell^s(\mathbb F_q^{d+1})} \le \mathcal{S}(p\to s)\|g\|_{\ell^p(\mathbb F_q^d)}$$
holds for all functions $g: \mathbb F_q^d \to \mathbb{C}.$ The $\mathcal{S}$-operator problem is to determine all exponents $1\le p, s\le \infty$  such that  $\mathcal{S}(p\to s) \le C$ for some constant $C$ independent of $q$.
\end{problem}

The primary motivation for studying the $\mathcal{S}$-operator problem in the paper \cite{KK14} was to investigate the existence of $\varepsilon$ in the spherical restriction estimate \eqref{HopeR}, which serves as a fundamental assumption in Corollary~\ref{epsilonIm} for even dimensions and enables us to break down the $(d+1)/2$ barrier in the Erd\H{o}s-Falconer distance problem in even dimensions $d \geq 4$. To explain the motivation more specifically, let us recall the following lemma.

\begin{lemma} [Lemma 3.4, \cite{KK26}] \label{Lem1.6} Let $1\le p,  r\le \infty$ and $C_1, C_2>0$. Suppose that  for all functions $g$ on $\mathbb F_q^d$, there exists  $1\le s\le \infty$ such that 
  $$\big\|\widehat{\mathcal{S}g}\big\|_{L^r(H_j^d)} ~\le C_1~\big\|\mathcal{S}g\big\|_{\ell^s(\mathbb F_q^{d+1})} \le C_2\|g\|_{\ell^p(\mathbb F_q^d)}.$$
Then we have
$$ R_{S_j^{d-1}}(p\to r)\lesssim 1.$$ 
\end{lemma}

From Lemma \ref{Lem1.6},  we see that if for each function $g:\mathbb F_q^d \to \mathbb C, $ there exists  $1\le s\le \infty$ such that  
\begin{equation}\label{HRes}\|\widehat{\mathcal{S}g}\|_{L^2(H_j^d)} ~\lesssim ~\|\mathcal{S}g\|_{\ell^{s} (\mathbb F_q^{d+1})} \end{equation}
and
\begin{equation}\label{SgProblem} \|\mathcal{S}g\|_{\ell^{s}(\mathbb F_q^{d+1})} \lesssim \|g\|_{\ell^{\frac{2d+4}{d+4} }(\mathbb F_q^d)},\end{equation}
then $R_{S_j^{d-1}}\left( \frac{2d+4}{d+4}\to 2  \right) \lesssim 1,$ which is the conjectured $L^2$ restriction result for the spheres in even dimensions $d\ge 4$ (see \cite{KPV}). The inequality appearing in \eqref{HRes} is the restriction problem related to $H_j^d$ for test functions with $F=\mathcal{S}g$ for the functions $g$ on $\mathbb{F}_q^d$, and we expect to address this problem in subsequent work. The inequality appearing in \eqref{SgProblem} is the main problem studied in this paper, namely the $\mathcal{S}$-operator problem.

\subsection{Statement of main results}
 An important and natural question is: does there exist $1 \le s \le \infty$ such that both inequalities \eqref{HRes} and \eqref{SgProblem} are satisfied for all functions $g: \mathbb{F}_q^d \to \mathbb{C}$? 
In this paper, we will show that no such $s$ independent of the function $g$ exists. To justify the nonexistence of such $s$, as our first main result, we establish a complete answer to the $S$-operator problem. As another main result of ours, we also prove a necessary and sufficient condition for the boundedness of the $\mathcal{S}$-operator in the case where the test function $g:\mathbb{F}_q^d \to \mathbb{C}$ is restricted to a radial function.

\subsubsection{Results on the $\mathcal{S}$-operator problem}
The test function $g$ is defined on $\mathbb{F}_q^d$, while $\mathcal{S}g$ is defined on $\mathbb{F}_q^{d+1}$. This makes the $\mathcal{S}$-operator distinct from operators commonly found in harmonic analysis. 
Our optimal result on the boundedness of the $\mathcal{S}$-operator is as follows, obtained by employing techniques adapted to the finite field setting. 

\begin{theorem} \label{MMain}
Let $1\le p, s\le \infty.$
Then $\mathcal{S}(p\to s)\lesssim 1$ if and only if the point $(1/p, 1/s)$ lies in the convex hull of the points $(0,0)$, $(1, 0)$, $(1, 1/2)$, and $(1/2, 1/2)$.
\end{theorem}

\begin{remark} We make the following observations regarding Theorem \ref{MMain}.
\begin{itemize}
\item The sharp result of Theorem \ref{MMain} is independent of the dimension $d$.
\item By virtue of the norm nesting property and the interpolation theorem (see Section \ref{secPre}), the proof of the sufficient condition for $\mathcal{S}(p\to s)\lesssim 1$ reduces to establishing the bounds $\mathcal{S}(\infty \to \infty)\lesssim 1$ and $\mathcal{S}(2\to 2)\lesssim 1$ (see Figure \ref{fig1}).
\item In \cite{KK26},  it was shown that if the test function $g$ belongs to a class of homogeneous functions of degree one, namely $g(tm)=g(m)$ for $t\ne 0,  m\in \mathbb F_q^d,$ then   the inequalities \eqref{HRes} and \eqref{SgProblem} hold with $s=(2d+4)/(d+4)$. However, Theorem \ref{MMain} shows that there exists a function $g$ for which \eqref{SgProblem} fails at $s=(2d+4)/(d+4)$.
\end{itemize}
\end{remark}

Invoking Theorem \ref{MMain},  we are able to prove the following:
\begin{corollary} \label{SadCor} There is no $1\le s\le \infty$ for which both inequalities \eqref{HRes} and \eqref{SgProblem} hold simultaneously for all functions $g: \mathbb{F}_q^d \to \mathbb{C}$
\end{corollary}

\begin{remark}  We emphasize that Corollary \ref{SadCor} does not imply the failure of the approach linking the boundedness of the $\mathcal{S}$-operator to improved restriction results for spheres, as Lemma 1.6 allows for the possibility that s exists depending on the function $g$.
\end{remark}

\begin{figure}[h]
\centering
\begin{tikzpicture}[scale=6]
% Draw coordinate axes with arrows
\draw[->] (0,0) -- (1.2,0) node[right] {$\frac{1}{p}$};
\draw[->] (0,0) -- (0,1.2) node[above] {$\frac{1}{s}$};

% Mark points on x-axis
\node[below] at (0.5,-0.05) {$\frac{1}{2}$};
\node[below] at (1,-0.05) {$1$};

% Mark points on y-axis
\node[left] at (-0.05,0.5) {$\frac{1}{2}$};
\node[left] at (-0.05,0.74) {$\frac{d}{d+1}$};
\node[left] at (-0.05,1) {$1$};

% Draw tick marks
\draw (0.5,-0.02) -- (0.5,0.02);
\draw (1,-0.02) -- (1,0.02);
\draw (-0.02,0.5) -- (0.02,0.5);
\draw (-0.02,0.74) -- (0.02,0.74);
\draw (-0.02,1) -- (0.02,1);

% Draw unit square outline with dashed lines
\draw[gray, dashed, thin] (0,1) -- (1,1);
\draw[gray, dashed, thin] (1,0) -- (1,1);

% Define the convex hull vertices
\coordinate (A) at (0,0);
\coordinate (B) at (1,0);  
\coordinate (C) at (1,0.74);
\coordinate (D) at (0.5,0.5);
\coordinate (E) at (1,0.5);

% Fill the convex hull region (red shaded area)
\fill[red!25] (A) -- (B) -- (C) -- (D) -- cycle;

% Fill the triangle formed by (1/2,1/2), (1,1/2), (1,d/(d+1)) with blue
\fill[blue!25] (D) -- (E) -- (C) -- cycle;

%\fill[green!25] (D) -- (A) -- (F) -- cycle;

% Draw the boundary of the convex hull
\draw[thick, black] (A) -- (B) -- (C) -- (D) -- cycle;

% Draw the boundary of the blue triangle
\draw[thick, black] (D) -- (E) -- (C) -- cycle;

% Mark the key vertices with filled circles
\fill (A) circle (0.8pt);
\fill (B) circle (0.8pt);
\fill (C) circle (0.8pt);
\fill (D) circle (0.8pt);
\fill (E) circle (0.8pt);

% Add vertex labels
\node[below left] at (A) {$(0,0)\ $};
\node[below right] at (B) {$\ (1,0)$};
\node[above right] at (C) {$\ (1,\frac{d}{d+1})$};
\node[above left] at (D) {$(\frac{1}{2},\frac{1}{2})$};
\node[right] at (E) {$\ (1,\frac{1}{2})$};

% Add the inequality label inside the red shaded region
\node[red, font=\footnotesize] at (0.6,0.3) { $\mathcal{S}(p \to s) \lesssim 1$};

% Add "Theorem 1.4" label below the inequality (moved down)
\node[red, font=\footnotesize] at (0.6,0.2) {(Theorem \ref{MMain})};

% Add "Conjecture" label inside the blue triangle
\node[blue, font=\tiny] at (0.815,0.535) {$\mathcal{S}_{\mathcal{R}_q}(p \to s) \lesssim 1$};

\end{tikzpicture}
\caption{ Sharp boundedness regions for the $\mathcal{S}$-operator (red shaded) and 
$\mathcal{S}_{\mathcal{R}_q}$-operator on radial functions (red + blue shaded). Boundaries are 
optimal: boundedness fails outside these regions. The radial restriction allows 
extension from $(1, 1/2)$ to $(1, d/(d+1))$.}
\label{fig1}
\end{figure}
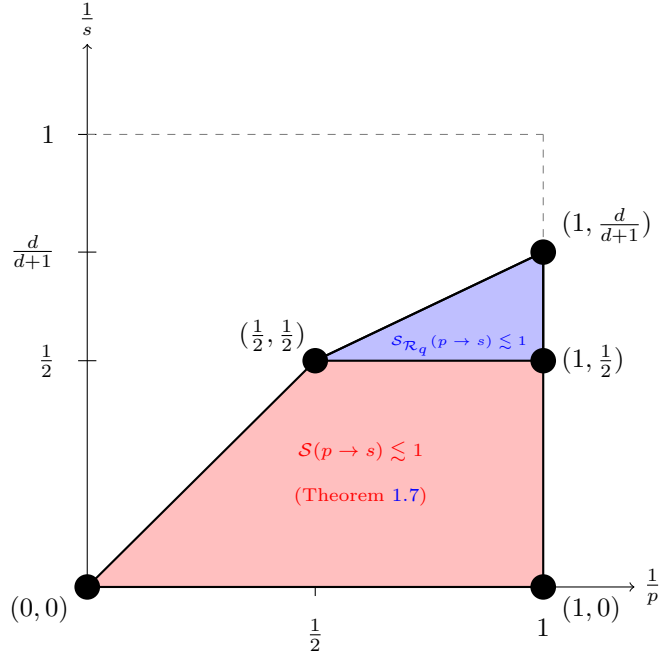

\subsubsection{Results on the restricted $\mathcal{S}$-operator problem}
Let $\mathcal{F}(\mathbb{F}_q^d \to \mathbb{C})$ denote the set of all functions from $\mathbb{F}_q^d$ to the complex numbers $\mathbb{C}$. For each finite field $\mathbb{F}_q$, we consider a subset $\mathcal{A}_q \subset \mathcal{F}(\mathbb{F}_q^d \to \mathbb{C})$. For $1\le p, s\le \infty,$ we define $\mathcal{S}_{\mathcal{A}_q}(p \to s)$ to be the smallest constant such that the inequality
$$
\|\mathcal{S}g\|_{\ell^s(\mathbb{F}_q^{d+1})} \le \mathcal{S}_{\mathcal{A}_q}(p \to s) \|g\|_{\ell^p(\mathbb{F}_q^d)}
$$
holds for all functions $g \in \mathcal{A}_q$. It is trivial that $\mathcal{S}_{\mathcal{A}_q}(p \to s) \le \mathcal{S}(p \to s)$. With this definition, we consider the following restricted $\mathcal{S}$-operator problem.
\begin{problem} [Restricted $\mathcal{S}$-operator problem to $\mathcal{A}_q$]   Find all numbers $1\le p, s\le \infty$ such that  $\mathcal{S}_{\mathcal{A}_q}(p \to s)\lesssim 1.$
\end{problem}

In \cite{KK26}, the authors showed that when $\mathcal{A}_q=\mathcal{H}_q$, which is the set of homogeneous functions of degree zero in $\mathbb{F}_q^d$, that is,
$
\mathcal{H}_q := \{g\in \mathcal{F}(\mathbb{F}_q^d\to \mathbb{C}): g(m)=g(\lambda m) \text{ for all } m\in \mathbb{F}_q^d, \lambda\in \mathbb{F}_q^*\},
$
we have 
\begin{equation} \label{ppestimate} \mathcal{S}_{\mathcal{A}_q}(p\to p) \lesssim 1 \quad  \mbox{for all}~~ 1\le p \le \infty. \end{equation}
 However, as seen in Theorem \ref{MMain}, there exists a class of test functions for which \eqref{ppestimate} fails. 

Now we are interested in proving sharp boundedness conditions for $\mathcal{S}_{\mathcal{A}_q} (p\to s)$  where  $\mathcal{A}_q  \subset \mathcal{F}(\mathbb{F}_q^d \to \mathbb{C}).$ As our result in this direction, we provide the answer in the case where $\mathcal{A}_q$ is the class of radial test functions. Recall that a function $g:\mathbb F_q^d \to \mathbb C$ is  radial  if  $g(m)=g(m')$ whenever $\|m\|=\|m'\|.$

In the case $(d, q)\ne (2,  3\pmod{4})$, we obtain the answer to the $\mathcal{S}$-operator problem restricted to radial test functions as follows (see Figure \ref{fig1}).
\begin{theorem} \label{mainRR}  Let $\mathcal{R}_q$ denote the class of all radial functions on $\mathbb F_q^d$  and let $1\le p, s\le \infty.$ Then the following two statements hold:
\begin{enumerate}
\item[(i)] If  $(1/p, 1/s)$ lies on the convex hull of the points $(0,0), (1, 0), (1, d/(d+1)),$ and $(1/2, 1/2),$ then $\mathcal{S}_{\mathcal{R}_q}\left(p \to s\right) \lesssim 1.$

\item[(ii)] In addition, if we assume that $(d, q)\ne (2, ~3\pmod{4}),$ then  the reverse statement of \text{(i)} is also true. 
\end{enumerate}
\end{theorem}

\subsection{Structure of the paper}
The remainder of this paper is organized as follows.
In Section~\ref{secPre}, we collect preliminary facts on the finite field Fourier transform, character orthogonality, sphere sizes, and interpolation theorems used throughout. Section~\ref{Sec3} provides the proof of Theorem~\ref{MMain}, establishing necessary and sufficient conditions for $S(p \to s) \lesssim 1$; the sufficient conditions are obtained by interpolating between the endpoint estimates $S(\infty \to \infty) \lesssim 1$ and $S(2 \to 2) \lesssim 1$, while the necessary conditions follow from testing against extremal functions. In Section \ref{Sec4}, we prove Corollary \ref{SadCor}, showing by contradiction that no uniform exponent $s$, independent of $g$, can satisfy both inequalities \eqref{HRes} and \eqref{SgProblem}. Finally, Section \ref{Sec5} is devoted to the proof of Theorem~1.12 on the restricted $\mathcal{S}$-operator problem for radial functions, where we establish the sharp exponent region via endpoint estimates and test functions of the form $g = \mathbf{1}_{S_j^{d-1}}$.

\section{Preliminaries} \label{secPre}
In this section, for the convenience of the reader, we summarize well-known facts that will be used to prove the main results of this paper. 

Let $\mathbb F_q$  be a finite field with $q=p^\alpha$, where $p$ is prime. Throughout this paper, we will use $\chi: \mathbb F_q \to \mathbb C^*$ to denote the principal additive character which is defined by
\[ \chi(s) = e^{2\pi i \text{Tr}(s)/p}, \]
where $\text{Tr}: \mathbb{F}_q \to \mathbb{F}_p$  is the trace function from $\mathbb F_q \to \mathbb{F}_p.$ We stress that the results of this paper hold for any choice of nontrivial additive character of $\mathbb F_q$.
%%%%%%%%%%%%%%%%%%%%%%%%%%%%%%%%%%%%%%%%%%%%%%%%%%%%%%%%%%%%%%%%%%%%%%%%%%%%%%%%%%%%%%%%

The most useful property of $\chi$ is the following orthogonality of non-trivial additive characters:
\begin{equation*}
\sum_{s \in \mathbb{F}_q} \chi(as) =
\begin{cases}
q & \text{if } a = 0, \\
0 & \text{if } a \neq 0.
\end{cases}
\end{equation*}
The Fourier transform of $f: \mathbb{F}_q^n \to \mathbb{C}$, denoted by $\widehat{f}$, is defined by
\begin{equation*}
\widehat{f}(x) = \sum_{m \in \mathbb{F}_q^n} f(m) \chi(-x\cdot m)
\end{equation*}
for $x \in \mathbb{F}_q^n$, where $x\cdot m= x_1m_1 + \cdots + x_n m_n$ denotes the standard inner product on $\mathbb{F}_q^n$.

The number of elements in  the sphere $S_j^{d-1}$ of radius $j$ in $\mathbb F_q^d$, denoted by $|S_j^{d-1}|,$  is given by
\begin{equation}\label{sizeS}\big|S_j^{d-1}\big| =\begin{cases} 
1 & \text{if }  j = 0,~ d=2,~ q\equiv 3\pmod{4}\\
\sim q^{d-1} & \text{otherwise.} 
\end{cases}\end{equation}

The norm nesting property states that  if $1\le p_0 \le p_1\le \infty,$ then 
\begin{equation}\label{NNP}\|f\|_{\ell^{p_1} (\mathbb F_q^n)} \le \|f\|_{\ell^{p_0} (\mathbb F_q^n)}.\end{equation}
Observe that  this norm nesting property implies that 
\begin{equation*}
 \mathcal{S}_{\mathcal{A}_q}(p \to s)  \le  \mathcal{S}_{\mathcal{A}_q}(p_1 \to s)   \quad \mbox{for}~~ 1\le p\le p_1\le\infty,
\end{equation*}
and
\begin{equation*}
 \mathcal{S}_{\mathcal{A}_q}(p\to s)  \le  \mathcal{S}_{\mathcal{A}_q}(p \to s_0)   \quad \mbox{for}~~ 1\le s_0\le s\le\infty.
\end{equation*}

From these observations, we can reduce the $\mathcal{S}$-operator problem to finding the largest possible value of $p$ for each fixed $1\le s\le \infty$ and the smallest possible value of $s$ for each fixed $1\le p\le \infty$ that satisfy $\mathcal{S}_{\mathcal{A}_q}(p\to s)\lesssim 1$. In other words, our problem becomes a matter of proving the critical endpoint estimates.  

In the finite field setting, the Riesz-Thorin interpolation theorem  can be stated as follows:
\begin{theorem}[Riesz-Thorin Interpolation Theorem over finite fields] \label{Intp}
Let $T: \ell^{p_0}(\mathbb{F}_q^d) \to L^{q_0}(\mathbb{F}_q^{d+1})$ and $T: \ell^{p_1}(\mathbb{F}_q^d) \to \ell^{q_1}(\mathbb{F}_q^{d+1})$ be bounded linear operators with norms $M_0$ and $M_1$ respectively, where $1 \leq p_0, p_1, q_0, q_1 \leq \infty$. Then for any $0 < \theta < 1$, the operator $T: \ell^{p_\theta}(\mathbb{F}_q^d) \to \ell^{q_\theta}(\mathbb{F}_q^{d+1})$ is bounded with norm $M_\theta$ satisfying
\begin{equation}
M_\theta \leq M_0^{1-\theta} M_1^\theta,
\end{equation}
where
\begin{equation}
\frac{1}{p_\theta} = \frac{1-\theta}{p_0} + \frac{\theta}{p_1} \quad \text{and} \quad \frac{1}{q_\theta} = \frac{1-\theta}{q_0} + \frac{\theta}{q_1}.
\end{equation}
\end{theorem}

Notice that  the Riesz-Thorin interpolation theorem on finite fields is essentially the same as the classical version, since $\mathbb{F}_q^n$ is a finite measure space with the counting measure (or normalized counting measure).

\section{ Sharp result for $\mathcal{S}(p\to s)\lesssim 1$ (proof of Theorem \ref{MMain})} \label{Sec3}
In this section, we present the proof of Theorem \ref{MMain}, which provides a necessary and sufficient condition for $\mathcal{S}(p\to s)\lesssim 1$.
To complete the proof,  we  prove  two theorems below  in the following subsections.

\begin{theorem} [Sufficient conditions]\label{T1} Let $1\le p, s\le \infty.$
If  the point $(1/p, 1/s)$ lies in the convex hull of the points $(0,0)$, $(1, 0)$, $(1, 1/2)$, and $(1/2, 1/2)$, then $ \mathcal{S}(p\to s)\lesssim 1. $
\end{theorem}

\begin{theorem} [Necessary conditions] \label{T2} Let $1\le p, s\le \infty.$
If $\mathcal{S}(p\to s)\lesssim 1$, then  the point $(1/p, 1/s)$ lies in the convex hull of the points $(0,0)$, $(1, 0)$, $(1, 1/2)$, and $(1/2, 1/2)$.
\end{theorem}

\subsection{Proof of Theorem \ref{T1}} We prove the results for the sufficient conditions for $\mathcal{S}(p\to s)\lesssim 1.$ Using the nesting properties of the norm  in \eqref{NNP}, it suffices to prove the following endpoint estimates.
\begin{theorem}\label{mainlem02}  For $2\le p\le \infty,$   the inequality
$$ \|\mathcal{S}_g\|_{\ell^p(\mathbb F_q^{d+1})} \lesssim \|g\|_{\ell^p(\mathbb F_q^d)}$$
holds for all functions $g$ on $\mathbb F_q^d.$
\end{theorem} 
\begin{proof}  Recall that the $\mathcal{S}$-operator is defined by
\[
\mathcal{S}g(m, m_{d+1}) = \frac{1}{q} \sum_{t \in \mathbb{F}_q^*} \chi(t m_{d+1}) g(tm)
\]
for $m \in \mathbb{F}_q^d$ and $m_{d+1} \in \mathbb{F}_q$. 

First, it is obvious that
\begin{equation}\label{pinfty2} \|\mathcal{S}g\|_{\ell^\infty(\mathbb F_q^{d+1})} \lesssim \|g\|_{\ell^\infty(\mathbb F_q^d)},\end{equation}
since, for all functions $g$ on $\mathbb F_q^d,$ we have 
$$\|\mathcal{S}g\|_{\ell^\infty(\mathbb F_q^{d+1})}:=\max\limits_{M\in \mathbb F_q^{d+1}} |\mathcal{S}g(M)| \le  \max\limits_{m\in \mathbb F_q^d} |g(m)|=\|g\|_{\ell^\infty(\mathbb F_q^d)}.$$ 

%%%%%%%%%%%%%%%%%%%%%%%%%%%%%%%%%%%%%%%%%%%%%%%%%%%%%%%%%%%%%%%%%%%%%%%%%%%%%%%%%%%%%%%%%%%%%%%%%%
Now we prove the theorem in the case where $p=2$. 
%We prove that $\|\mathcal{S}g\|_{\ell^2(\mathbb{F}_q^{d+1})} \lesssim \|g\|_{\ell^2(\mathbb{F}_q^d)}$ for all functions $g : \mathbb{F}_q^d \to \mathbb{C}$. 
We compute
\begin{align*}
\|\mathcal{S}g\|_{\ell^2(\mathbb{F}_q^{d+1})}^2 
&= \sum_{m \in \mathbb{F}_q^d} \sum_{m_{d+1} \in \mathbb{F}_q} |\mathcal{S}g(m, m_{d+1})|^2 \\
&= \sum_{m \in \mathbb{F}_q^d} \sum_{m_{d+1} \in \mathbb{F}_q} \frac{1}{q^2} \left| \sum_{t \in \mathbb{F}_q^*} \chi(t m_{d+1}) g(tm) \right|^2.
\end{align*}
For each fixed $(m, m_{d+1}) \in \mathbb{F}_q^d \times \mathbb{F}_q$, we expand the square:
\begin{align*}
\left| \sum_{t \in \mathbb{F}_q^*} \chi(t m_{d+1}) g(tm) \right|^2
&= \sum_{ s, t \in \mathbb{F}_q^* } \chi(m_{d+1}(s - t)) g(sm) \overline{g(tm)}.
\end{align*}
Substituting this back, we obtain
\[
\|\mathcal{S}g\|_{\ell^2(\mathbb{F}_q^{d+1})}^2 
= \frac{1}{q^2} \sum_{m \in \mathbb{F}_q^d} \sum_{m_{d+1} \in \mathbb{F}_q} \sum_{ s, t \in \mathbb{F}_q^* } \chi(m_{d+1}(s - t)) g(sm) \overline{g(tm)}.
\]
By the orthogonality relation of the character $\chi$, we have
\[
\sum_{m_{d+1} \in \mathbb{F}_q} \chi(m_{d+1}(s - t)) = \begin{cases}
q & \text{if } s = t, \\
0 & \text{if } s \neq t.
\end{cases}
\]
Therefore, only terms with $s = t$ survive:
\begin{align*}
\|\mathcal{S}g\|_{\ell^2(\mathbb{F}_q^{d+1})}^2 
&= \frac{1}{q} \sum_{m \in \mathbb{F}_q^d} \sum_{ s \in \mathbb{F}_q^* }  |g(sm)|^2 .
\end{align*}
By Fubini's theorem, we can change the order of summation:
\begin{align*}
\|\mathcal{S}g\|_{\ell^2(\mathbb{F}_q^{d+1})}^2 
&= \frac{1}{q} \sum_{ s \in \mathbb{F}_q^* } \sum_{m \in \mathbb{F}_q^d}  |g(sm)|^2.
\end{align*}
For a fixed $s\in \mathbb F_q^*,$  we substitutes $\alpha=sm$.  
Hence,
\[
\|\mathcal{S}g\|_{\ell^2(\mathbb{F}_q^{d+1})}^2 
= \frac{1}{q} \sum_{ s \in \mathbb{F}_q^* } \sum_{\alpha \in \mathbb{F}_q^d} |g(\alpha)|^2 =\frac{q-1}{q} \|g\|_{\ell^2(\mathbb F_q^d)}^2.
\]
Taking square roots on both sides, we conclude that
\begin{equation} \label{ptwo2}
\|\mathcal{S}g\|_{\ell^2(\mathbb{F}_q^{d+1})} \lesssim \|g\|_{\ell^2(\mathbb{F}_q^d)}.
\end{equation}
Interpolating the estimates,  \eqref{ptwo2} and  \eqref{pinfty2}, we complete the proof of Theorem \ref{mainlem02}.
\end{proof}

\subsection{Proof of Theorem \ref{T2}}  Assume that for $1\le p, s\le \infty,$ 
$$ L:=\|\mathcal{S}g\|_{\ell^s(\mathbb F_q^{d+1})} \lesssim \|g\|_{\ell^p(\mathbb F_q^d)}=:R,$$
which holds for all function $g:\mathbb F_q^d \to \mathbb C.$
In other words,  for all functions $g: \mathbb F_q^d \to \mathbb C,$ it satisfies that
\begin{equation}\label{AssE} L= \left( \sum_{m\in \mathbb F_q^d, m_{d+1}\in \mathbb F_q} |\mathcal{S}g(m, m_{d+1})|^s\right)^{\frac{1}{s}} \lesssim  \left(\sum_{m\in \mathbb F_q^d} |g(m)|^p\right)^{\frac{1}{p}}=R.\end{equation}

\noindent \textbf{CASE 1}. Let us test the inequality \eqref{AssE} with $g=\delta_{x_0},$
where $x_0 \in \mathbb F_q^d \setminus \{\mathbf{0}\}$ and  $\delta_{x_0}(x)=1$ for $x=x_0$ and  0  otherwise. 

It is clear that 
$$ R=1.$$

For $m \in \mathbb F_q^d$ and $m_{d+1}\in \mathbb F_q,$  we note that
$$ \mathcal{S}g(m, m_{d+1})= \frac{1}{q} \sum_{t\in \mathbb F_q^*} \chi(tm_{d+1}) \delta_{x_0}(tm) =\frac{1}{q} \sum_{t\in \mathbb F_q^*: x_0=tm} \chi(tm_{d+1}).$$
For $m \neq \mathbf{0}$, if $x_0 = t_0 m$ for some unique $t_0 \in \mathbb{F}_q^*$, then
$$\mathcal{S}g(m, m_{d+1}) = q^{-1} \chi(t_0 m_{d+1}).$$
Otherwise, $\mathcal{S}g(m, m_{d+1}) = 0$. 
Let $\ell_{x_0}^*$ denote the line through the origin and $x_0$, with the origin removed. Then we have
$$ |\mathcal{S}g(m, m_{d+1})|=\begin{cases} 
q^{-1} & \text{if }  m\in \ell^*_{x_0},  m_{d+1}\in \mathbb F_q,\\
0 & \text{otherwise.} 
\end{cases}$$
Hence, we obtain that
$$ L= \left(\sum_{m\in \ell^*_{x_0}, m_{d+1} \in \mathbb F_q} q^{-s}\right)^{\frac{1}{s}} \sim  q^{\frac{2-s}{s}}.$$
Since $L \lesssim R,$  we conclude that
$\frac{2-s}{s} \le 0,$  which yields  a necessary condition that
\begin{equation} \label{Go1}  \frac{1}{s}\le \frac{1}{2}.\end{equation}

\noindent \textbf{CASE 2}. We test the inequality \eqref{AssE} with $g=1_V$, where $V$ denotes a subspace with dimension $k\ge 1$, lying in $\mathbb F_q^d.$

It is not hard to see that
$$ R=|V|^{\frac{1}{p}}.$$

On the other hand, since $V$ is a subspace,  
$$\mathcal{S}g(m, m_{d+1})= \frac{1}{q} \sum_{t\in \mathbb F_q^*} \chi(tm_{d+1}) 1_{V}(tm) =\frac{1}{q} \sum_{t\in \mathbb F_q^*: m\in V} \chi(tm_{d+1}).$$
Using the orthogonality of $\chi$, this implies that
$$ |\mathcal{S}g(m, m_{d+1})|=\begin{cases} 
\frac{q-1}{q} & \text{if }  m\in V,  m_{d+1}=0,\\
q^{-1} & \text{if }  m\in V,  m_{d+1}\ne 0,\\
0 & \text{otherwise.} 
\end{cases}$$
Therefore, we obtain that
$$ L=\left( \sum_{m\in V, m_{d+1}=0}  \left(\frac{q-1}{q}\right)^s + \sum_{m\in V, m_{d+1}\in \mathbb F_q^*} q^{-s}\right)^{\frac{1}{s}}
\sim \left( |V| + |V| q^{1-s}\right)^{\frac{1}{s}}.$$
Since $s\ge 1,$ it follows that
$$ L\sim  |V|^{\frac{1}{s}}.$$
Since $L\lesssim R$ and $|V|\gtrsim q,$   we obtain another necessary  condition that
$$ \frac{1}{s} \le \frac{1}{p}.$$
Since $0\le \frac{1}{p}, \frac{1}{s}\le 1$,  combining the above condition together with \eqref{Go1} gives the conclusion of Theorem \ref{T2}, as required.

\section{Proof of Corollary \ref{SadCor} } \label{Sec4}
We will complete the proof using Proof by Contradiction. To this end, assume that there exists a constant $1 \le s = s_0 \le \infty$ satisfying \eqref{HRes} and \eqref{SgProblem} that is independent of all functions $g: \mathbb{F}_q^d \to \mathbb{C}$. That is, assume there exist uniform constants $C_1, C_2 > 0$ such that the following two inequalities hold for all functions $g: \mathbb{F}_q^d \to \mathbb{C}$:

\begin{equation}\label{HRes*}\|\widehat{\mathcal{S}g}\|_{L^2(H_j^d)} \le C_1\|\mathcal{S}g\|_{\ell^{s_0} (\mathbb F_q^{d+1})} \end{equation}
and
\begin{equation}\label{SgProblem*} \|\mathcal{S}g\|_{\ell^{s_0}(\mathbb F_q^{d+1})} \le C_2\|g\|_{\ell^{\frac{2d+4}{d+4} }(\mathbb F_q^d)}.\end{equation}

From Theorem \ref{MMain} and \eqref{SgProblem*}, we must have $2 \le s_0 \le \infty$. By the norm nesting property in \eqref{NNP},   it therefore follows that
$$\|\mathcal{S}g\|_{\ell^{s_0} (\mathbb F_q^{d+1})} \le \|\mathcal{S}g\|_{\ell^{2} (\mathbb F_q^{d+1})}. $$
Combining this inequality with \eqref{HRes*},  it must follow that for all functions $g$ on $\mathbb F_q^d,$
$$\|\widehat{\mathcal{S}g}\|_{L^2(H_j^d)} \le C_1\|\mathcal{S}g\|_{\ell^{2} (\mathbb F_q^{d+1})}. $$
In particular,  when $g_0(m):=\chi(a\cdot m)$ for a fixed $a\in S_j^{d-1} \subset \mathbb F_q^d \setminus \{\mathbf{0}\},$ it satisfies that
\begin{equation}\label{False}\|\widehat{\mathcal{S}g_0}\|_{L^2(H_j^d)} \le C_1\|\mathcal{S}g_0\|_{\ell^{2} (\mathbb F_q^{d+1})}. \end{equation}

We will complete the proof of  Corollary \ref{SadCor} by showing that the inequality \eqref{False} cannot hold. Since $\mathcal{S}g_0(m, m_{d+1})=q^{-1} \sum_{t\in \mathbb F_q^*} \chi(t(m_{d+1}+ a\cdot m))$ for $m\in \mathbb F_q^d$, $m_{d+1}\in \mathbb F_q,$  it follows by the orthogonality of $\chi$  that
$$ \mathcal{S}g_0(m, m_{d+1}) = \begin{cases} 
\frac{q-1}{q} & \text{if }   m_{d+1}+a\cdot m=0,\\
-q^{-1} & \text{if }  m_{d+1}+a\cdot m\ne 0.
\end{cases}$$
Form this observation,  we get
$$\|\mathcal{S}g_0\|_{\ell^{2} (\mathbb F_q^{d+1})} =\left(\sum_{\substack{m\in \mathbb F_q^d, m_{d+1}\in \mathbb F_q\\: m_{d+1}+ a\cdot m=0}} \left(\frac{q-1}{q}\right)^2 + \sum_{\substack{m\in \mathbb F_q^d, m_{d+1}\in \mathbb F_q\\: m_{d+1}+ a\cdot m\ne 0}} q^{-2}\right)^{\frac{1}{2}}.$$
Since $|\{(m, m_{d+1})\in \mathbb F_q^d \times \mathbb F_q: m_{d+1}+ a\cdot m=0\}|=q^d $   and $|\{(m, m_{d+1})\in \mathbb F_q^d \times \mathbb F_q: m_{d+1}+ a\cdot m\ne 0\}|= q^{d+1}-q^d \sim q^{d+1},$ we obtain that
\begin{equation}\label{Sade1}
\|\mathcal{S}g_0\|_{\ell^{2} (\mathbb F_q^{d+1})} \sim   ( q^d + q^{d-1})^{\frac{1}{2}} \sim  q^{\frac{d}{2}}.
\end{equation}

Now let us estimate the left-hand side $\|\widehat{\mathcal{S}g_0}\|_{L^2(H_j^d)}$ of the inequality \eqref{False}.
In the proof of Lemma 3.4 in \cite{KK26}, the following was proved:
$$\|\widehat{\mathcal{S}g_0}\|_{L^2(H_j^d)} \sim \|\widehat{g_0}\|_{L^2(S_j^{d-1})}.$$
By the orthogonality of $\chi$,  we see that
$$ \widehat{g_0}(x)=\begin{cases} 
0 & \text{if }   x\ne a,\\
q^d & \text{if }  x=a.
\end{cases}$$
Since $j\ne 0$ and $a\in S_j^{d-1},$  we have
$$ \|\widehat{\mathcal{S}g_0}\|_{L^2(H_j^d)} \sim \left(\frac{1}{|S_j^{d-1}|} \sum_{x\in S_j^{d-1}} |\widehat{g_0}(x)|^2\right)^{\frac{1}{2}}\sim  \left(\frac{1}{q^{d-1}}  q^{2d} \right)^{\frac{1}{2}} $$
Hence, we obtain that
$$\|\widehat{\mathcal{S}g_0}\|_{L^2(H_j^d)} \sim q^{\frac{d+1}{2}}. $$
From this estimate and \eqref{Sade1},  we conclude that the inequality \eqref{False} cannot be true. 

\section{The $\mathcal{S}$-operator problem for the radial test functions (proof of Theorem \ref{mainRR})} \label{Sec5}
In this section, we provide a complete proof of Theorem \ref{mainRR}.
We need to prove that the following two statements are true.

\begin{theorem}[Sufficient conditions] \label{TT1} 
 If  $(1/p, 1/s)$ lies on the convex hull of the points\\
 $(0,0), (1, 0), (1, d/(d+1)),$ and $(1/2, 1/2),$ then $\mathcal{S}_{\mathcal{R}_q}\left(p \to s\right) \lesssim 1.$ \end{theorem}

\begin{theorem} [Necessary conditions]\label{TT2} If $(d, q)\ne (2, ~3\pmod{4})$ and $\mathcal{S}_{\mathcal{R}_q}\left(p \to s\right) \lesssim 1$,  then $(1/p, 1/s)$ lies on the convex hull of the points $(0,0), (1, 0), (1, d/(d+1)),$ and $(1/2, 1/2).$
\end{theorem}

\subsection{Proof of Theorem \ref{TT1}} Since $\mathcal{S}_{\mathcal{R}_q}\left(p \to s\right) \le \mathcal{S}\left(p \to s\right),$  Theorem \ref{MMain} implies that  if $(1/p, 1/s)$ lies on the convex hull of the points $(0,0), (1, 0), (1, 1/2),$ and $(1/2, 1/2),$ then $\mathcal{S}_{\mathcal{R}_q}\left(p \to s\right) \lesssim 1.$ Therefore,  to complete the proof of Theorem \ref{TT1},  it suffices by the interpolation theorem  to show that  
\begin{equation*}\mathcal{S}_{\mathcal{R}_q}\left(1 \to \frac{d+1}{d}\right)  \lesssim 1.\end{equation*}
In other words,  for any radial functions $g: \mathbb F_q^d \to \mathbb C,$  our task is to prove that
$$ \|\mathcal{S}g\|_{\ell^{\frac{d+1}{d}}(\mathbb F_q^{d+1})}\lesssim \|g\|_{\ell^1(\mathbb F_q^d)}.$$
Since the operator $\mathcal{S}$ is a linear operator, without loss of generality,  we may assume that the function $g:\mathbb F_q^d\to \mathbb C$ is a radial function with nonnegative real values. Furthermore,  by normalizing the function $g$, we also assume that 
$$\|g\|_{\ell^1(\mathbb F_q^d)}=1.$$
Letting $g(m)=M_j\ge 0$ if $\|m\|=j\in \mathbb F_q,$  the above condition becomes
$$ 1=\sum_{m\in \mathbb F_q^d} |g(m)|= \sum_{j\in \mathbb F_q} \sum_{m\in S_j^{d-1}} M_j =\sum_{j\in \mathbb F_q} |S_j^{d-1}| M_j.$$ 
Invoking the size for the spheres $S_j^{d-1}$ in \eqref{sizeS}, this condition is equivalent to 
\begin{equation} \label{Aim1} 1  \sim \begin{cases} M_0 +
q \sum\limits_{j\in \mathbb F_q^*} M_j &\text{if }  d=2,~ q\equiv 3\pmod{4}\\
q^{d-1} \sum\limits_{j\in \mathbb F_q} M_j &\text{otherwise.} 
\end{cases}\end{equation}
Under this assumption, our aim is to show that
$$ \|\mathcal{S}g\|_{\ell^{\frac{d+1}{d}}(\mathbb F_q^{d+1})}\lesssim 1. $$
Thus, the proof reduces to showing that under the assumption \eqref{Aim1}
\begin{equation}  
L:= \|\mathcal{S}g\|^{ \frac{d+1}{d}}_{\ell^{\frac{d+1}{d}}(\mathbb F_q^{d+1})} = \sum_{m\in \mathbb F_q^d, m_{d+1}\in \mathbb F_q} \left| \frac{1}{q} \sum_{t\in \mathbb F_q^*} \chi(tm_{d+1}) g(tm)\right|^{\frac{d+1}{d}} \lesssim 1.
\end{equation}
Notice that $ \mathbb F_q^d=\bigcup_{k\in \mathbb F_q} S_k^{d-1}$ and $\|tm\|=t^2k$ for $t\in \mathbb F_q^*, m\in S_k^{d-1}.$ It follows that
\begin{equation}\label{FirstCase} L=q^{-\frac{d+1}{d}} \sum_{m_{d+1}\in \mathbb F_q} \sum_{k\in \mathbb F_q} |S_k^{d-1}| \left|\sum_{t\in \mathbb F_q^*} \chi(tm_{d+1}) M_{t^2k}\right|^{\frac{d+1}{d}}.\end{equation}

\noindent \textbf{CASE 1}. Assuming that $(d, q)\ne (2,~3\pmod{4}),$ we show that  $L\lesssim 1.$ Note from \eqref{sizeS} that  in this case,
 $|S_k^{d-1}|\sim q^{d-1}$ for all $k\in \mathbb F_q$.   Hence, it follows that
$$ L \sim q^{-\frac{d+1}{d}} q^{d-1} \sum_{m_{d+1}, k\in \mathbb F_q} \left|\sum_{t\in \mathbb F_q^*} \chi(tm_{d+1}) M_{t^2k}\right|^{\frac{d+1}{d}}.$$
Using H\"{o}lder's inequality,  we obtain that
\begin{align*} L&\lesssim q^{-\frac{d+1}{d}} q^{d-1 } \left(\sum_{m_{d+1}, k\in \mathbb F_q} 1^{\frac{2d}{d-1}}\right)^{\frac{d-1}{2d}}  \left(\sum_{m_{d+1}, k\in \mathbb F_q}\left|\sum_{t\in \mathbb F_q^*} \chi(tm_{d+1}) M_{t^2k}\right|^2\right)^{\frac{d+1}{2d}}\\
&= q^{-\frac{d+1}{d}} q^{d-1} q^{\frac{d-1}{d}} \left(\sum_{m_{d+1}, k\in \mathbb F_q}\sum_{t, t'\in \mathbb F_q^*} \chi(m_{d+1}(t-t')) M_{t^2k} M_{t'^2 k}\right)^{\frac{d+1}{2d}}.\end{align*}
Now, computing the sum over the variable $m_{d+1}\in \mathbb F_q$ by the orthogonality of $\chi$,  it follows that
$$ L\lesssim q^{-\frac{d+1}{d}} q^{d-1} q^{\frac{d-1}{d}} q^{\frac{d+1}{2d}} \left(\sum_{k\in \mathbb F_q}\sum_{t\in \mathbb F_q^*} M_{t^2k}^2\right)^{\frac{d+1}{2d}}.$$
By a simple change of variables, it is clear that $\sum_{k\in \mathbb F_q}\sum_{t\in \mathbb F_q^*} M_{t^2k}^2=(q-1) \sum_{k\in \mathbb F_q}M_{k}^2.$
It therefore follows that
\begin{equation} \label{MKE}L\lesssim q^{\frac{d^2-1}{d}} \left(\sum_{k\in \mathbb F_q}M_{k}^2\right)^{ \frac{d+1}{2d}}.\end{equation}
From \eqref{Aim1}, we may assume that  $\sum_{k\in \mathbb F_q} M_k\sim q^{-d+1}$ and thus $M_k\lesssim q^{-d+1}$ for all $k\in \mathbb F_q.$ Using these conditions, we observe that
$$\sum_{k\in \mathbb F_q}M_{k}^2 \lesssim q^{-d+1} q^{-d+1} =q^{-2d+2}.$$
Combining this observation and \eqref{MKE},  we obtain the desired result,  $L\lesssim 1.$\\

\noindent\textbf{CASE 2}. Assume $d=2$ and $q\equiv 3\pmod{4}.$ Let us prove that $L\lesssim 1.$  In this case,  it follows by  \eqref{sizeS} and \eqref{FirstCase} that 
\begin{align*} L&\sim  q^{-\frac{3}{2}} \sum_{m_{3}\in \mathbb F_q}  \left|\sum_{t\in \mathbb F_q^*} \chi(tm_{3}) M_{0}\right|^{\frac{3}{2}}  + q^{-\frac{1}{2}} \sum_{m_{3}\in \mathbb F_q} \sum_{k\in \mathbb F_q^*}  \left|\sum_{t\in \mathbb F_q^*} \chi(tm_{3}) M_{t^2k}\right|^{\frac{3}{2}}\\
&=:I + II.\end{align*}
It suffices to show that $I\lesssim 1$  and $II\lesssim 1.$
Since $II \lesssim 1$ is easily proven by exactly the same process as in Case 1, we omit the details of the proof.  

Since $\chi(0)=1$, by using the orthogonality of $\chi$,  we obtain that
\begin{align*} I&=q^{-\frac{3}{2}} M_{0}^{\frac{3}{2}}\sum_{m_{3}\in \mathbb F_q}  \left|\sum_{t\in \mathbb F_q^*} \chi(tm_{3}) \right|^{\frac{3}{2}} \\
&= q^{-\frac{3}{2}} M_{0}^{\frac{3}{2}} (q-1)^{\frac{3}{2}} 
+q^{-\frac{3}{2}} M_{0}^{\frac{3}{2}} (q-1) \sim  M_{0}^{\frac{3}{2}},\end{align*}
where  we computed separately for the cases $m_3 = 0$ and $m_3 \neq 0$.
 Using \eqref{Aim1} with the assumption of Case 2, we see that  $1\sim M_0 + q \sum_{k\in \mathbb F_q^*} M_k.$  This implies $M_0\lesssim 1$.
Therefore, we obtain the required result, $I\lesssim 1$.
 
\subsection{Proof of Theorem \ref{TT2}}
We begin with the following lemma. 
\begin{lemma} \label{OSC} For any non-negative integer $n$,  we have
$$ \sum_{t\in \mathbb F_q} | \chi(t) + \chi(-t)|^n \sim q.$$ 
\end{lemma}
\begin{proof}
By the interpolation theorem,  it suffices to show that  for any even number $n$,  
$$ \sum_{t\in \mathbb F_q} | \chi(t) + \chi(-t)|^n \sim q.$$ 
Since $\chi(t)+\chi(-t)$ is a real number,   we see that  
$|\chi(t)+\chi(-t)|^n= (\chi(t)+\chi(-t))^n $  for any even integer $n.$
Hence, by the binomial theorem,  it follows that
\begin{align*}  \sum_{t\in \mathbb F_q} | \chi(t) + \chi(-t)|^n &= \sum_{t\in \mathbb F_q} ( \chi(t) + \chi(-t))^n\\
&=\sum_{t\in \mathbb F_q} \sum_{i=0}^n   {n\choose i} \chi^i(t)  \chi^{n-i}(-t) = \sum_{i=0}^n {n\choose i} \sum_{t\in \mathbb F_q} \chi((2i-n)t) = {n\choose n/2} q, \end{align*}
as required.
\end{proof} 

We now start proving Theorem \ref{TT2}. 
Assume that for $1\le p, s\le \infty,$ 
$$ L:=\|\mathcal{S}g\|_{\ell^s(\mathbb F_q^{d+1})} \lesssim \|g\|_{\ell^p(\mathbb F_q^d)}=:R,$$
where $g$ is any radial function on $\mathbb F_q^d.$
This is the same as the following:
$$ L= \left( \sum_{m\in \mathbb F_q^d, m_{d+1}\in \mathbb F_q} \left| \frac{1}{q} \sum_{t\in \mathbb F_q^*} \chi(tm_{d+1}) g(tm)\right|^s\right)^{\frac{1}{s}} \lesssim  \left(\sum_{m\in \mathbb F_q^d} |g(m)|^p\right)^{\frac{1}{p}}=R.$$
It is clear that  $1_{S_j^{d-1}}$,  the indicator function of the sphere with a radius $j\in \mathbb F_q,$ is a radial function on $\mathbb F_q^d.$
We test this inequality with  $g=1_{S_j^{d-1}}.$  
Then it follows that
\begin{equation*} R=|S_j^{d-1}|^{\frac{1}{p}}.\end{equation*}
On the other hand, it follows that
\begin{equation} \label{KLDef} L= \frac{1}{q} \left(\sum_{m\in \mathbb F_q^d, m_{d+1}\in \mathbb F_q } \left| \sum_{t\in \mathbb F_q^*:t^2\|m\|=j} \chi(m_{d+1}t) \right|^s\right)^{\frac{1}{s}}.\end{equation}
We assume that $(d, q)\ne (2, ~3\pmod{4}).$  Then we see from \eqref{sizeS} that 
$|S_j^{d-1}|\sim q^{d-1}$ for all $j\in \mathbb F_q.$ Hence,  we have
\begin{equation} \label{R} R\sim q^{\frac{d-1}{p}}.
\end{equation}
\noindent\textbf{CASE 1}. We consider the case where  $j=0$ in \eqref{KLDef}. Then we have
$$ L=\frac{1}{q} \left(\sum_{m\in \mathbb F_q^d, m_{d+1}\in \mathbb F_q: \|m\|= 0 } \left| \sum_{t\in \mathbb F_q^*} \chi(m_{d+1}t) \right|^s\right)^{\frac{1}{s}}$$
$$= \frac{1}{q} \left(|S_0^{d-1}|\sum_{ m_{d+1}\in \mathbb F_q} \left| q \delta_0(m_{d+1})-1 \right|^s\right)^{\frac{1}{s}}.$$
Since $|S_0^{d-1}|\sim q^{d-1}$, 

$$ L= q^{\frac{d-1}{s}-1} \left( (q-1)^s+ \sum_{m_{d+1}\ne 0} 1\right)^{\frac{1}{s}}=q^{\frac{d-1}{s}-1} \left( (q-1)^s+ q-1 \right)^{\frac{1}{s}}.$$
Hence, for $ s\ge 1$, we see that
$L \sim q^{\frac{d-1}{s}}.$
Since $L\lesssim R,$  this estimate and \eqref{R} yield that if $(d, q)\ne (2, ~3\pmod{4})$, then a necessary condition for $\mathcal{S}(p\to s)\lesssim 1$ is
\begin{equation} \label{Nec1}  1\le p\le s\le \infty. \end{equation}

\noindent\textbf{CASE 2}. We take $j=1$ in \eqref{KLDef}. Then we have
$$ L=\frac{1}{q} \left(\sum_{m\in \mathbb F_q^d, m_{d+1}\in \mathbb F_q: \eta(\|m\|)=1 }  
\left| \chi\left(\frac{m_{d+1}}{\sqrt{\|m\|}}\right) +   \chi\left(-\frac{m_{d+1}}{\sqrt{\|m\|}}\right)\right|^s\right)^{\frac{1}{s}},  $$
where $\eta: \mathbb{F}_q \to \{-1, 0, 1\}$ denotes the quadratic character defined by
$$\eta(t) = \begin{cases}
\ \ 1 & \text{if } t \text{ is a square in } \mathbb{F}_q^*, \\
-1 & \text{if } t \text{ is not a square in } \mathbb{F}_q^*,\\
\ \ 0 & \text{if } t=0.
\end{cases}$$
Observe that  $ |\{m\in \mathbb F_q^d: \eta(\|m\|)=1\}|\sim q^{d}.$
Since $\|m\|\ne 0$ for $m\in \mathbb F_q^d$ with $\eta(\|m\|)=1,$ we are able to apply a change of variables by replacing $m_{d+1}$ by $\sqrt{\|m\|} m_{d+1}.$ It follows that
$$ L= \frac{1}{q} \left( \sum_{m\in \mathbb F_q^d: \eta(\|m\|)=1} \sum_{m_{d+1}\in \mathbb F_q} \left| \chi(m_{d+1})+ \chi(-m_{d+1})\right|^s\right)^{\frac{1}{s}}$$
$$\sim q^{\frac{d}{s}-1} \left(\sum_{m_{d+1}\in \mathbb F_q} \left| \chi(m_{d+1})+ \chi(-m_{d+1})\right|^s\right)^{\frac{1}{s}} \sim q^{\frac{d+1}{s}-1},$$
where the last estimate follows immediately from Lemma  \ref{OSC}.
Since $L\lesssim R,$   the above estimate and \eqref{R} give a necessary condition for $\mathcal{S}(p\to s)\lesssim 1$ as follows: 
\begin{equation}\label{Nec2} \frac{d+1}{s}-\frac{d-1}{p}\le 1.\end{equation}

The proof of Theorem \ref{TT2}  is completed by direct calculation from the results of \eqref{Nec1} and \eqref{Nec2}.

\backmatter

\bmhead{Acknowledgements}
D. Koh and C. Yang were supported by the National Research Foundation of Korea(NRF) grant funded by the Korea government(MSIT) (NO.~RS-2023-00249597) and (No.~2021R1C1C1005700), respectively.

\end{document}